\documentclass[11pt,a4paper]{article}

\usepackage{amsmath,amssymb,amsthm}
\usepackage{mathtools}
\usepackage{geometry}
\usepackage{booktabs}
\usepackage{enumitem}
\usepackage[dvipsnames]{xcolor}
\usepackage{hyperref}
\hypersetup{colorlinks=true,linkcolor=MidnightBlue,citecolor=MidnightBlue,urlcolor=MidnightBlue}

\newtheorem{theorem}{Theorem}
\newtheorem{proposition}[theorem]{Proposition}

\theoremstyle{definition}

\newtheorem{example}[theorem]{Example}
\theoremstyle{remark}
\newtheorem{remark}[theorem]{Remark}

\newcommand{\R}{\mathbb{R}}
\newcommand{\C}{\mathbb{C}}
\newcommand{\eps}{\varepsilon}
\newcommand{\dbar}{\bar{\partial}}
\newcommand{\im}{\operatorname{Im}}
\newcommand{\re}{\operatorname{Re}}

\title{A Degenerate Elliptic System Solvable by Transport:\\
A Cautionary Example}
\author{Daniel Alay\'on-Solarz\thanks{danieldaniel@gmail.com}}
\date{\today}

\begin{document}
\maketitle

\begin{abstract}
We exhibit a one-parameter family of first-order real elliptic systems on the plane whose
ellipticity constant degenerates to zero as $\delta\to 0$,
with condition number $\kappa = O(\delta^{-2})$.
For any fixed elliptic solver operating at finite precision,
the parameter $\delta$ can be chosen small enough to
defeat the solver;
no uniform numerical scheme based on the ellipticity constant alone
can handle the entire family.
Despite this, every member of the family is explicitly solvable---and its
initial value problem well posed for analytic data---by elementary means once
a transport-theoretic invariant is identified.
The cost of the transport solution is independent of~$\delta$.
The example serves as a cautionary tale: the ellipticity constant alone
does not determine the practical difficulty of a first-order PDE.
Before invoking an elliptic solver, one should compute the
transport obstruction~$G$; its vanishing---or smallness---signals
structure that standard elliptic methods miss entirely.
\end{abstract}

\section{The system}\label{sec:system}

Fix a parameter $\delta > 0$ and consider the real first-order system
for $(u,v) \in C^1(\Omega, \R^2)$, where $\Omega = \{(x,y) \in \R^2 : x > -1\}$:
\begin{equation}\label{eq:system}
\begin{cases}
\displaystyle u_x - \frac{y^2 + \delta^2}{(1+x)^2}\, v_y = 0, \\[10pt]
\displaystyle v_x + u_y + \frac{2y}{1+x}\, v_y = 0.
\end{cases}
\end{equation}

This is a homogeneous, linear, first-order system with smooth coefficients on~$\Omega$.
One may verify directly that its principal symbol has nonzero determinant
for every nonzero covector, so the system is elliptic on all of~$\Omega$.

We denote the coefficients
\begin{equation}\label{eq:alphabeta}
\alpha(x,y) = \frac{y^2 + \delta^2}{(1+x)^2}, \qquad
\beta(x,y) = \frac{-2y}{1+x},
\end{equation}
so that \eqref{eq:system} reads
\begin{equation}\label{eq:system-short}
\begin{cases}
u_x - \alpha\, v_y = 0, \\
v_x + u_y - \beta\, v_y = 0.
\end{cases}
\end{equation}

\begin{remark}[No domain restriction on~$y$]
Unlike many explicit examples of variable-coefficient elliptic systems,
the elliptic domain here is the entire half-plane $\{x > -1\}$, with no
restriction on~$y$. This removes any concern about boundary effects.
\end{remark}

\section{Degeneracy}\label{sec:degen}

\subsection{Principal symbol and ellipticity constant}

Write \eqref{eq:system-short} in the standard form
$P(x,y,\partial)\mathbf{u} = 0$
where
$\mathbf{u} = (u, v)^T$ and $P$ has the principal symbol matrix
\[
P(\xi,\eta) =
\begin{pmatrix}
\xi & -\alpha\,\eta \\
\eta & \xi - \beta\,\eta
\end{pmatrix}.
\]
Its determinant is
$\det P(\xi,\eta) = \xi^2 - \beta\,\xi\eta + \alpha\,\eta^2$.
Completing the square,
\[
\det P = \bigl(\xi - \tfrac{\beta}{2}\eta\bigr)^2
+ \tfrac{1}{4}\bigl(4\alpha - \beta^2\bigr)\eta^2.
\]
The ellipticity discriminant is
\begin{equation}\label{eq:discrim}
\Delta(x,y) := 4\alpha - \beta^2 = \frac{4\delta^2}{(1+x)^2}.
\end{equation}
Since $\delta > 0$ and $1+x > 0$ on $\Omega$, we have $\Delta > 0$ everywhere, confirming
ellipticity.

\subsection{The Beltrami coefficient}

The standard reduction of a first-order elliptic system to Beltrami form
introduces the spectral parameter
\[
\lambda = \frac{-\beta + i\sqrt{\Delta}}{2} = \frac{y + i\delta}{1+x},
\]
and the associated Beltrami coefficient
\[
\mu := \frac{\lambda - i}{\lambda + i}
= \frac{y + i(\delta - 1 - x)}{y + i(\delta + 1 + x)}.
\]
Its squared modulus is
\begin{equation}\label{eq:mu}
|\mu|^2 = \frac{y^2 + (\delta - 1 - x)^2}{y^2 + (\delta + 1 + x)^2}.
\end{equation}
At the origin,
$|\mu(0,0)| = |1 - \delta|/(1 + \delta)$,
and on any compact subset $K \subset \Omega$,
\begin{equation}\label{eq:mu-uniform}
\inf_K |\mu_\delta| \;\longrightarrow\; 1
\qquad \text{as } \delta \to 0^+.
\end{equation}

Table~\ref{tab:mu} records the range of $|\mu|$ on the set
$K = [-\tfrac{1}{2},\, 1] \times [-1,\, 1]$, together with
the elliptic condition number
$\kappa := \bigl(\tfrac{1+\|\mu\|_\infty}{1-\|\mu\|_\infty}\bigr)^2$,
which governs the conditioning of the equivalent second-order problem.

\begin{table}[h]
\centering
\begin{tabular}{@{}ccccc@{}}
\toprule
$\delta$ & $\displaystyle\inf_K |\mu|$ & $\displaystyle\sup_K |\mu|$
& $\kappa$ & Char.\ cost \\
\midrule
$1$      & $0$    & $0.620$ & $18$                & $O(1)$ \\
$10^{-1}$   & $0.667$ & $0.924$ & $6.3\times 10^{2}$  & $O(1)$ \\
$10^{-2}$  & $0.961$ & $0.992$ & $6.3\times 10^{4}$  & $O(1)$ \\
$10^{-3}$ & $0.996$ & $0.999$ & $6.3\times 10^{6}$  & $O(1)$ \\
$10^{-4}$ & $0.9996$ & $0.9999$ & $6.3\times 10^{8}$ & $O(1)$ \\
\bottomrule
\end{tabular}
\caption{Degeneration of the $\delta$-family on $K = [-\tfrac{1}{2}, 1]\times[-1,1]$.
The elliptic condition number $\kappa$ grows as $O(\delta^{-2})$,
while the characteristic method (Section~\ref{sec:char}) has
cost independent of~$\delta$.}\label{tab:mu}
\end{table}

\subsection{Failure of standard elliptic methods}\label{sec:failure}

The degeneration \eqref{eq:mu-uniform} is not merely aesthetic.
Standard computational methods for elliptic systems of Beltrami type
encounter quantitative failure as $\delta \to 0$.

\smallskip
\noindent\textbf{Beurling--Ahlfors iteration.}
The standard approach to the Beltrami equation $w_{\bar z} = \mu\, w_z$ is
to invert $I - \mu S$, where $S$ is the Beurling--Ahlfors singular integral
operator with $\|S\|_{L^p \to L^p} = p^* - 1$ for the optimal~$p$.
The Neumann series converges when $\|\mu\|_\infty \|S\|_{L^p} < 1$.
As $\delta \to 0$, the contraction factor tends to~$1$ and the series diverges.

\smallskip
\noindent\textbf{Direct solvers.}
Newton-type methods for the nonlinear Beltrami problem require inverting
$I - \mu S$ at each step, with condition number scaling as
$\kappa \sim 1/(1 - \|\mu\|_\infty) \to \infty$.

\smallskip
\noindent\textbf{Finite elements.}
The associated divergence-form operator has coefficient matrix
with eigenvalue ratio
$(1 - |\mu|)/(1 + |\mu|) \to 0$, so the resulting linear systems
become increasingly ill-conditioned; cf.\ Table~\ref{tab:mu}.

\medskip

This degeneracy is \emph{uniform on compact sets}:
it is not a boundary layer phenomenon
or a localized singularity, but a global deterioration of
the ellipticity constant.

\begin{remark}[Arbitrary degeneracy]\label{rem:arbitrary}
The family defeats any fixed elliptic method by parameter tuning alone:
for any solver operating at fixed arithmetic precision~$\eps$,
there exists $\delta_0(\eps) > 0$ below which the condition number
$\kappa \sim \delta^{-2}$ exceeds $\eps^{-1}$ and the method fails,
while the characteristic solution (Section~\ref{sec:char}) has cost
independent of~$\delta$.
\end{remark}

\begin{remark}[Relation to the degenerate Beltrami literature]
There is a substantial body of work on the Beltrami equation in the
degenerate regime $\|\mu\|_\infty = 1$, concerned primarily with
finding optimal integrability conditions on the dilatation
$K_\mu = (1+|\mu|)/(1-|\mu|)$ that guarantee existence, uniqueness,
and regularity of homeomorphic solutions;
see~\cite{GMSV2005} for directional dilatation criteria
and~\cite{GRSY2012,AIM2009} for comprehensive treatments.
The present example lives squarely within this degenerate regime
as $\delta \to 0$.
However, the questions addressed are complementary:
the existing theory asks
\emph{under what conditions do solutions exist despite the degeneracy},
while our example asks
\emph{whether the degeneracy is intrinsic or merely an artifact
of the Beltrami formulation}.
The transport obstruction~$G$ provides an answer invisible to
the dilatation: when $G = 0$, the system is explicitly solvable
regardless of~$K_\mu$, and the apparent degeneracy reflects not
a genuine analytical difficulty but rather the unsuitability of
the Beltrami parametrization for the problem at hand.
\end{remark}

\section{The rigidity approach}\label{sec:rigid}

We now show that system~\eqref{eq:system}
has a hidden transport structure
that makes it explicitly solvable, regardless of
the value of~$\delta$.

\subsection{The transport obstruction}\label{sec:obstruction}

System~\eqref{eq:system-short} is the Cauchy--Riemann system
associated with a variable elliptic structure defined by the polynomial
$X^2 + \beta(x,y)\, X + \alpha(x,y) = 0$,
with $\alpha, \beta$ as in \eqref{eq:alphabeta}.  Its generator
$\mathfrak{i}(x,y)$ satisfies
$\mathfrak{i}^2 + \beta\,\mathfrak{i} + \alpha = 0$
in each fiber, and an algebra-valued function $f = u + v\,\mathfrak{i}$ is holomorphic
($\dbar f = 0$, where $\dbar = \tfrac{1}{2}(\partial_x + \mathfrak{i}\,\partial_y)$)
precisely when $(u,v)$ satisfies \eqref{eq:system-short}---\emph{provided}
a certain \emph{transport obstruction} vanishes.

More precisely, expanding $\dbar f = 0$ using
$\mathfrak{i}^2 = -\beta\,\mathfrak{i} - \alpha$ yields the identity
\begin{equation}\label{eq:expansion}
2\,\dbar f = \bigl(u_x - \alpha\,v_y\bigr)
+ \bigl(v_x + u_y - \beta\,v_y\bigr)\,\mathfrak{i}
\;+\; v\,G,
\end{equation}
where
\begin{equation}\label{eq:obstruction}
G := \mathfrak{i}_x + \mathfrak{i}\,\mathfrak{i}_y
\end{equation}
is the \emph{intrinsic obstruction} of the structure~\cite[\S2.2]{AS2026}.
(The derivatives $\mathfrak{i}_x$, $\mathfrak{i}_y$ are determined by implicit
differentiation of $\mathfrak{i}^2 + \beta\,\mathfrak{i} + \alpha = 0$.)

When $G \neq 0$, the equation $\dbar f = 0$ forces additional
lower-order coupling through the $vG$ term,
and system~\eqref{eq:system-short} is \emph{not} the correct
Cauchy--Riemann system.
When $G = 0$, the $vG$ term drops out, equation~\eqref{eq:expansion}
shows that $\dbar f = 0$ is \emph{exactly}
equivalent to~\eqref{eq:system-short},
and the operator $\dbar$ satisfies the Leibniz rule.
This is the \emph{rigid} regime.

\subsection{Rigidity of the \texorpdfstring{$\delta$}{delta}-family}

\begin{proposition}\label{prop:rigid}
For every $\delta > 0$, the transport obstruction~\eqref{eq:obstruction} vanishes
identically: $G \equiv 0$ on~$\Omega$.
\end{proposition}

\begin{proof}
The vanishing $G = 0$ is equivalent~\cite[Theorem~2.3]{AS2026} to the
conservative Burgers equation $\lambda_x + \lambda\,\lambda_y = 0$
for the spectral parameter.
A direct computation:
\[
\lambda_x = -\frac{y + i\delta}{(1+x)^2}, \qquad
\lambda_y = \frac{1}{1+x}, \qquad
\lambda_x + \lambda\,\lambda_y
= -\frac{y + i\delta}{(1+x)^2} + \frac{y + i\delta}{(1+x)^2}
= 0. \qedhere
\]
\end{proof}

\begin{remark}[Checking rigidity from $\alpha$, $\beta$ alone]
One does not need the spectral parameter to check rigidity.
The obstruction $G$, expressed in the basis $\{1, \mathfrak{i}\}$
as $G = A + B\,\mathfrak{i}$, has coefficients~\cite[\S1.6]{AS2026}
\[
A = \frac{\beta(\alpha_x - \alpha\beta_y) - 2\alpha(\beta_x + \alpha_y - \beta\beta_y)}{\Delta},
\quad
B = \frac{2(\alpha_x - \alpha\beta_y) - \beta(\beta_x + \alpha_y - \beta\beta_y)}{\Delta}.
\]
Substituting \eqref{eq:alphabeta} and simplifying gives $A = B = 0$.
This check requires only $\alpha$, $\beta$, and their first partial derivatives.
\end{remark}

\subsection{From the real system to the spectral transport equation}\label{sec:spectral}

We now derive the spectral transport equation that makes the system
explicitly solvable.  The derivation is self-contained from
system~\eqref{eq:system-short}.

Write the spectral parameter in real and imaginary parts:
\begin{equation}\label{eq:lambda-ab}
\lambda = a + ib, \qquad a(x,y) = \frac{y}{1+x}, \quad
b(x,y) = \frac{\delta}{1+x}.
\end{equation}
Then $\alpha = a^2 + b^2$ and $\beta = -2a$.

\begin{proposition}[Spectral transport form]\label{prop:spectral}
Let $(u,v)$ be a $C^1$ solution of the real system~\eqref{eq:system-short}.
Define the complex-valued function
\begin{equation}\label{eq:w-def}
w := u + v\,\lambda = (u + av) + i\,bv.
\end{equation}
Then $w$ satisfies the transport equation
\begin{equation}\label{eq:transport}
w_x + \lambda\, w_y = 0.
\end{equation}
Conversely, if $w = p + iq$ satisfies~\eqref{eq:transport}, then
$u = p - (a/b)\,q$ and $v = q/b$ satisfy~\eqref{eq:system-short}.
\end{proposition}

\begin{proof}
Compute, using $w = u + v\lambda$:
\[
w_x + \lambda\,w_y
= \bigl(u_x + \lambda\,u_y\bigr)
+ \lambda\bigl(v_x + \lambda\,v_y\bigr)
+ v\bigl(\lambda_x + \lambda\,\lambda_y\bigr).
\]
The last term vanishes by Proposition~\ref{prop:rigid}.
For the remaining terms, expand and collect real parts
using $\alpha = a^2 + b^2$ and $\beta = -2a$:
\begin{align*}
\re\bigl(u_x {+} \lambda u_y {+} \lambda(v_x {+} \lambda v_y)\bigr)
&= u_x + au_y + av_x + (a^2{-}b^2)v_y.
\end{align*}
From~(I): $u_x = (a^2{+}b^2)v_y$.
From~(II): $v_x = -u_y - 2av_y$. Substituting:
$(a^2{+}b^2)v_y + au_y + a(-u_y - 2av_y) + (a^2{-}b^2)v_y = 0$.
For the imaginary part:
$\im(\cdots) = b(u_y + v_x + 2av_y)$,
which vanishes by~(II).

For the converse, set $u = p - (a/b)q$ and $v = q/b$,
note that $b_y = 0$ for our family, and reverse the computation.
\end{proof}

\begin{remark}[The limit $\delta \to 0$]
The converse recovery $v = q/b$ requires $b = \delta/(1+x) \neq 0$,
which holds for every $\delta > 0$ but fails at $\delta = 0$.
At $\delta = 0$ the structure becomes parabolic ($\Delta = 0$),
the spectral parameter $\lambda = y/(1+x)$ is real,
and the map $(u,v) \mapsto w = u + v\lambda$ ceases to be invertible.
The family is thus genuinely elliptic for each $\delta > 0$,
but the spectral identification degenerates in the limit.
\end{remark}

\begin{remark}
Equation~\eqref{eq:transport} is a linear,
complex-valued transport equation
with coefficient $\lambda(x,y)$---a known, explicit function of
$(x,y)$, determined by the coefficients of the original system,
not by the solution~$w$.
The map $(u,v) \mapsto w = u + v\lambda$ is the \emph{spectral identification}:
it recombines the real unknowns using the spectral parameter as
a mixing coefficient, absorbing the variable fiber structure
into the standard complex plane~$\C$.
\end{remark}

\subsection{Explicit solutions and the characteristic method}\label{sec:char}

Since $\lambda$ satisfies $\lambda_x + \lambda\lambda_y = 0$,
it is constant along the characteristic lines $y - \lambda x = \text{const}$.
The conserved quantity
\begin{equation}\label{eq:zeta}
\zeta(x,y) := y - x\,\lambda(x,y) = \frac{y - i\delta x}{1+x}
\end{equation}
is the \emph{characteristic coordinate}.
Any analytic function of~$\zeta$ solves the transport
equation~\eqref{eq:transport}.

\begin{proposition}[Solution by characteristics]\label{prop:char}
Let $f_0$ be analytic on a neighborhood of $\{x = 0\}$ in $\C$.
Then the unique solution of~\eqref{eq:transport} with
$w(0,y) = f_0(y)$ is
\begin{equation}\label{eq:char-sol}
w(x,y) = f_0\!\bigl(\zeta(x,y)\bigr)
= f_0\!\biggl(\frac{y - i\delta x}{1+x}\biggr),
\end{equation}
valid in a neighborhood of $\Gamma = \{x=0\}$ determined by the
characteristics.
\end{proposition}

\begin{proof}
On $\Gamma$: $\zeta(0,y) = y$, so $w(0,y) = f_0(y)$.
Along characteristics: $\zeta$ is constant, hence $f_0(\zeta)$ is
constant, so $w_x + \lambda w_y = 0$.
Uniqueness follows from the noncharacteristic transversality
of $\Gamma$.
\end{proof}

\begin{example}[Explicit real solution pair]\label{ex:real}
Take the power $w = \lambda^2$.  By Proposition~\ref{prop:spectral},
this corresponds to the real pair
$\mathfrak{i}^2 = -\alpha - \beta\,\mathfrak{i}$, giving
\[
u = -\alpha = -\frac{y^2+\delta^2}{(1+x)^2}, \qquad
v = -\beta = \frac{2y}{1+x}.
\]
Direct verification:
$(u)_x = 2(y^2{+}\delta^2)/(1{+}x)^3 = \alpha\cdot 2/(1{+}x) = \alpha\,(v)_y$;
$(v)_x + (u)_y - \beta\,(v)_y = -2y/(1{+}x)^2 - 2y/(1{+}x)^2 + 4y/(1{+}x)^2 = 0$.
\end{example}

\begin{example}[An initial value problem]\label{ex:ivp}
With $f_0(y) = e^{y + i\delta}$, Proposition~\ref{prop:char} gives
\[
w(x,y)
= \exp\!\biggl(\frac{y + i\delta}{1+x}\biggr) = e^{\lambda},
\]
verified by $w_x + \lambda w_y
= e^\lambda\bigl(-\lambda/(1{+}x) + \lambda/(1{+}x)\bigr) = 0$.
\end{example}

\begin{remark}[Cost comparison]\label{rem:cost}
The characteristic solution~\eqref{eq:char-sol} requires
evaluating $\zeta(x,y)$ (a rational function of $x$, $y$, $\delta$)
and composing with $f_0$.  At each grid point, this is a
bounded number of arithmetic operations---the same for
$\delta = 1$ and $\delta = 10^{-10}$.
By contrast, an elliptic finite-element solver on an $N\times N$ grid
must invert a system whose condition number scales as
$\kappa \sim \delta^{-2}$
(Table~\ref{tab:mu}).
For $\delta = 10^{-4}$, this is $\kappa \sim 10^{8}$.

The characteristic method is not, of course, a universal numerical
scheme: it requires analytic initial data (or data that can be
analytically continued into the complex $\zeta$-plane), and its
domain of validity is bounded by the characteristic geometry.
The point is that when the transport structure is present,
it provides a solution mechanism whose cost is \emph{entirely decoupled}
from the ellipticity constant.

The characteristic domain deserves a brief comment.
The real characteristic curves of the transport equation
$w_x + \lambda w_y = 0$ are the integral curves of
$dy/dx = \re\lambda = y/(1+x)$, namely $y = C(1+x)$ for $C \in \R$.
The complex conserved quantity $\zeta = (y - i\delta x)/(1+x)$
restricts to a constant along each such curve;
it simultaneously provides the analytic continuation of the
initial data into the complex $\zeta$-plane.
Starting from initial data on $\Gamma = \{x=0\}$,
the characteristics remain non-intersecting in a neighborhood whose size
depends on $f_0$ but not on~$\delta$.
\end{remark}

\subsection{Why the computation is \texorpdfstring{$\delta$}{δ}-independent}\label{sec:why}

The resolution of the apparent paradox is as follows.
The Beltrami coefficient $\mu$ measures how far the
complex structure deviates from the standard one.
As $\delta \to 0$, the structure deviates maximally, and
any method based on deforming from the standard structure
(which is what Beurling--Ahlfors iteration does) must
traverse nearly the full diameter of Teichm\"uller space.

But the system has additional structure that $\mu$ does not see:
the transport obstruction $G = \mathfrak{i}_x + \mathfrak{i}\,\mathfrak{i}_y$
vanishes.  This means that the spectral parameter $\lambda$
satisfies the conservative Burgers equation
$\lambda_x + \lambda\,\lambda_y = 0$,
so the coefficients $(\alpha, \beta)$ are governed by
transport along characteristics rather than
by an elliptic fixed-point problem.

The vanishing of $G$ is invisible to the Beltrami coefficient.
Two systems with the same $|\mu|$ profile can have
completely different transport obstructions: one may be rigid
($G = 0$) and explicitly solvable, while the other may
require the full force of elliptic regularity theory.

\section{Conclusion}\label{sec:conclusion}

The $\delta$-family \eqref{eq:system} is an explicit one-parameter family of
smooth, real, first-order elliptic systems on the half-plane with
the following properties:
\begin{enumerate}[label=(\roman*),nosep]
\item the coefficients $\alpha$, $\beta$ are rational functions of
$(x, y, \delta)$;
\item the system is elliptic for every $\delta > 0$, but its
ellipticity constant degenerates to zero as $\delta \to 0$,
with condition number $\kappa = O(\delta^{-2})$;
\item standard elliptic solvers---Beurling--Ahlfors iteration,
Newton--Beltrami methods, finite elements---deteriorate quantitatively as
$\delta \to 0$;
\item yet the system is explicitly solvable via the characteristic
method (Proposition~\ref{prop:char}), at cost independent of~$\delta$.
\end{enumerate}

The mechanism behind (iv) is the vanishing of the transport obstruction
$G = \mathfrak{i}_x + \mathfrak{i}\,\mathfrak{i}_y$
(Proposition~\ref{prop:rigid}).
This is a first-order, pointwise, algebraically computable
quantity that depends only on $\alpha$, $\beta$,
and their first derivatives.  Its vanishing is equivalent to saying
that the spectral parameter satisfies a
conservative transport law, which in turn implies that solutions
propagate along characteristics.

To our knowledge, the degenerate Beltrami literature
(e.g.~\cite{GMSV2005,AIM2009,GRSY2012})
does not contain an example that is simultaneously
arbitrarily ill-conditioned for all standard elliptic solvers
and explicitly solvable by an alternative method at
$\delta$-independent cost.
The examples studied there are designed to probe the
boundary of the existence and uniqueness theory
(integrability conditions on the distortion~$K_\mu$),
not to expose a gap between elliptic conditioning
and practical solvability.

The intermediate regime---$G \not\equiv 0$ but $\|G\|$ small,
where one expects a perturbative transition between
characteristic and elliptic methods---is not addressed here.

\medskip

The practical lesson is this:
\emph{before applying an elliptic solver to a first-order smooth PDE system,
one should compute the transport obstruction.}
If $G$ vanishes---or is small relative to the other scales
in the problem---transport-based methods
(characteristic tracing, similarity factorization,
integrating factors along characteristics, variable elliptic analysis) may succeed where
elliptic methods stall.
The obstruction $G$ is cheap to evaluate: it requires only the
first partial derivatives of the coefficients, combined in a
specific algebraic way.
Including it as part of the numerical triage
is a negligible cost with
potentially decisive benefit.

One might summarize the situation as follows.
The Beltrami coefficient $\mu$ measures the distance from the standard
complex structure.
The transport obstruction $G$ measures whether
the structure moves \emph{coherently}---compatibly with its own
multiplication law.
A structure can be far from standard ($|\mu| \approx 1$)
and yet coherent ($G = 0$).  When it is,
the difficulty is a mirage: the problem lives on a
submanifold of the Beltrami space where characteristic
methods provide what elliptic methods cannot.

\section*{Acknowledgements}

\noindent\textbf{Use of Generative AI Tools.}
Portions of the writing and editing of this manuscript were assisted by generative AI language tools. These tools were used to improve clarity of exposition, organization of material, and language presentation. All mathematical results, statements, proofs, and interpretations were developed, verified, and validated by the author. The author takes full responsibility for the accuracy, originality, and integrity of all content in this work, including any material produced with the assistance of AI tools. No generative AI system is listed as an author of this work.

\end{document}